\documentclass[11pt, letterpaper]{article}

\usepackage[portrait,margin=3cm]{geometry}
\usepackage{preamble}
\begin{document}

\title{Fluctuation of free energy for the mean-field Ghatak-Sherrington Model}
\author{
    Yueqi Sheng\thanks{School of Engineering \& Applied Sciences, Harvard University, Cambridge, Massachusetts, USA.
    \newline Email: \textcolor{red}{ysheng@g.harvard.edu}.}
}
\date{}

\maketitle
\begin{abstract}
	In this short note, we prove a central limit theorem for the free energy in the Ghatak-Sherrington model at high enough temperatures, based on a generalization of the stochastic calculus method for the SK model derived in \cite{comets95, Tindel_2005}.
\end{abstract}


\section{Introduction}
The Ghatak-Sherrington Model is defined as follows: the spin configurations are in 
\[\Sigma_{S, N} = \{ 0, \pm 1, \cdots, \pm S\}^N\]
The Hamiltonian is given by 
\[H_{N}(\sigma) = \frac{\beta}{\sqrt{N}} \sum_{i < j}g_{i, j}\sigma_i\sigma_j + D\sum_i \sigma_i^2 + h \sum_i \sigma_i\]
where $\beta>0$ is the inverse temperature, and $h\ge0$ and $D\in \mathbb{R}$ represent the \textit{external} and \textit{crystal} fields respectively. The corresponding Gibbs measure is then defined as $\frac{\exp(H_N(\sigma))}{Z_N}$, where $Z_N$ is the partition function.  The free energy is 
\[\frac{1}{N}\log Z_N = \frac{1}{N}\log \left(\sum_\sigma \exp(H_N(\sigma))\right)\]
The stochastic calculus method has been used to derive fluctuation results of free energy for the SK model ~\cite{comets95, Tindel_2005} (See also \cite{Talagrand2011}). The main idea of the proof is to view the interpolation in the cavity method as a combination of two Brownian motions running in opposite directions. A CLT for free energy could then be obtained by applying Ito's lemma. For the GS model, similar methods bring out terms of the form $U_{k, l}$ from the quadratic variation of the cavity process. This is because self-overlap is no longer trivial. In this note, we show that an analogous result for the GS model holds despite having non-trivial self-overlap. The main result is the following central limit theorem of the free energy.
\begin{theorem} \label{thm: main}
There exists $\beta'$ s.t for $\beta \in [0, \beta']$ s.t.
    \[\sqrt{N}\left(\frac{1}{N} \log Z_N - \frac{1}{N}\E[\log Z_0] - \frac{\beta^2}{4}(q^2 - p^2)\right)\]
    converges to $\cc{N}(0, \nu)$ as $N \to \infty$, where $p, q$ are constants defined in Proposition \ref{prop:ac21}, and the variance is given by
    \[\nu^2 = \left(\Var(\log W) - \frac{\beta^2 q^2}{2}\right)\]
    where $W = \sum_{\gamma = -S}^S \exp((\beta \sqrt{q}\eta + h)\gamma + \gamma^2 \left(\frac{\beta^2}{2}(p - q) + D\right))$ and $\eta \sim \cc{N}(0, 1)$
\end{theorem}
An explicit formal for $\frac{1}{N} \log Z_0$ can be found in Lemma \ref{lem: Free energy 0}.
\section{Set up}
We first rewrite cavity interpolation as a stochastic process. Let $\{B_{i, j}[t]: 1 \leq i < j \leq N\}$ and $\{B_i[t]: 1 \leq i \leq N\}$ be two independent family of Brownian motion in $(\Omega, \cc{F}_t, P)$ and $\cc{F}_t$ measurable. 
\begin{lemma}[\cite{Tindel_2005, IW81}] \label{lem: time-rev} 
Let $W_i$ be the unique solution to the following SDE
    \[dW_i[t] = -\frac{W_i[t]}{1 - t}dt + dB_i[t] \quad W_i[0] \sim \cc{N}(0, 1)\]
    Then $W_i[t] = B_i[1 - t]$ can be seen as a time reversed BM for $t \in [0, 1]$. 
\end{lemma}
\textit{In the remainder of this note, we write $H_t$ as $H_{N, t}$.}

To describe the interpolation, recall that the overlap between two copies of spin configuration $\sigma^k, \sigma^l$ is defined as
\[R_{k, l} = \frac{1}{N}\inner{\sigma^k, \sigma^l}\]
At sufficiently high temperatures, it is expected that the overlap and self-overlap concentrate on some fixed points respectively. The explicit form of the system of equations and the following concentration results were given in ~\cite{AC21}.
\begin{proposition} [{\cite[Proposition 2]{AC21}}] \label{prop:ac21}
    There exist a $\tilde{\beta} > 0$ s.t for $\beta \in [0,  \tilde{\beta})$, $h \geq 0$ and $\cc{D} \in \R$, there exists unique $p, q \in \mathbf{R}$ s.t.
    \[\nu((R_{1,2} -q)^2) \le \frac{16\cS^2}{N}, \quad 
    \nu((R_{1,1} -p)^2) \le \frac{16\cS^4}{N}.\]
\end{proposition}
For convenience, let us also define the analog for overlap on the square of spin configuration.
\[U_{k, l} = \frac{1}{N^2} \sum_i (\sigma^k_i)^2(\sigma^l_i)^2\]
The interpolated Hamiltonian can be defined as follows: 
Let $\{\eta_i\}$ be a family of independent Gaussian defined on the same probability space. Let $\eta_i = W_i[0]$
\begin{align} \label{eq:init Ham}
    H_0(\sigma) = \beta \sqrt{q}\sum_i \sigma_i \eta_i + \frac{\beta^2}{2} (p - q) \sum_i \sigma_i^2 + D \sum_i \sigma_i^2 + h\sum_i \sigma_i
\end{align}
\[dH_t(\sigma) = \frac{\beta}{\sqrt{N}} \sum_{i < j}\sigma_i\sigma_j dB_{i, j}[t] + \beta \sqrt{q} \sum_i \sigma_i dW_i[t] - \frac{\beta^2}{2} (p - q) \sum_i \sigma_i^2 dt\]
Observe that at time $1$, $H_t = H_N$. 

We denote $Z_t$ as the partition function corresponding to the Hamiltonian $H_t$, and $\inner{\cdot}_t$ as the Gibbs average at time $t$.

\section{Proof of the main theorem}
In this section, we prove the main result of this note.
\begin{theorem}[rewrite of Theorem \ref{thm: main}]
There exists $\beta'$ s.t for $\beta \in [0, \beta']$
    \[\E[\exp(i u \sqrt{N}F_1] = \exp(-\frac{u^2 \nu^2}{2}) + O(\frac{1}{\sqrt{N}})\]
where the centered (normalized) free energy at time $t$ is defined as
    \[N^{\frac{1}{2}}F_t = \log Z_t - \E[\log Z_0] - \frac{\beta^2N}{4}(q^2 - p^2)t\]
and
    \[\nu^2 = \left(\Var(\log W) - \frac{\beta^2 q^2}{2}\right)\]
    where $W = \sum_{\gamma = -S}^S \exp((\beta \sqrt{q}\eta + h)\gamma + \gamma^2 \left(\frac{\beta^2}{2}(p - q) + D\right))$ and $\eta \sim \cc{N}(0, 1)$
\end{theorem}
The first step is to derive $d\log Z_t$ (see Lemma \ref{lem: dlog}). First, let's introduce the following notation to simplify $dH_t$.
\[dX_t(\sigma) = \frac{\beta}{\sqrt{N}} \sum_{i < j}\sigma_i\sigma_j dB_{i, j}[t]\]
and 
\[dY_t(\sigma) =  \beta \sqrt{q} \sum_i \sigma_i dW_i[t]\]
where $dW_i$ is defined as in Lemma \ref{lem: time-rev}. 
Using the above definition,
\[dH_t(\sigma) = dX_t(\sigma) + dY_t(\sigma) - \frac{\beta^2}{2} (p - q) \sum_i \sigma_i^2 dt\]
so that 
\[\E[H_N(\sigma^1)H_N(\sigma^2)] = \frac{\beta^2}{2} N(R^2_{1, 2} - U_{1, 2})\]
When $i = j$, $R_{i, j}$ and $U_{i, j}$ becomes self overlaps.

The quadratic variation of $dH_t$ is given by the lemma below.
\begin{lemma}
    \[d[H(\sigma)]_t = \frac{\beta^2N}{2} (R^2_{1, 2} - U_{1, 2})dt + \beta^2 q N R_{1, 1}dt\]
\end{lemma}
\begin{proof}
    Since $B_{i, j}[t]$ and $B_i[t]$ are independent from each other, 
    \begin{align*}
        d[H(\sigma)]_t &= d[X(\sigma)]_t + d[Y(\sigma)]_t\\
        &= \frac{\beta^2N}{2} (R^2_{1, 1} - U_{1, 1})dt + \beta^2 q N R_{1, 1}dt
    \end{align*} 
\end{proof}
Next, we obtain $dZ_t$ and its quadratic variation 
\begin{lemma} \label{lem: partition fn}
    \begin{align*}
        dZ_t &= Z_t \inner{dX_t(\sigma) + dY_t(\sigma)}_t\\
        &+ Z_t \inner{ - \frac{\beta^2}{2} (p - q) \sum_i \sigma_i^2 dt + \frac{\beta^2N}{4} (R^2_{1, 1} - U_{1, 1})dt + \frac{\beta^2N}{2} q R_{1, 1}dt}_t
    \end{align*}
    and
    \begin{align*}
         d[Z]_t &=Z_t^2 \inner{\frac{\beta^2N}{2}(R^2_{1, 2} - U_{1, 2})dt + \beta^2 q N R_{1, 2}}_t
    \end{align*}
\end{lemma}
\begin{proof}
    Apply Ito's lemma to $f(x) = \exp(x)$ to $H_t(\sigma)$, we have 
    \begin{align*}
        d\exp(H_t(\sigma)) &= \exp(H_t(\sigma)) \left(dH_t(\sigma) + \frac{1}{2}d[H(\sigma)]_t\right)\\
        &= \exp(H_t(\sigma))\left(dX_t(\sigma) + dY_t(\sigma) - \frac{\beta^2}{2} (p - q) \sum_i \sigma_i^2 dt\right)\\
        &+ \exp(H_t(\sigma))\left(\frac{\beta^2N}{4} (R^2_{1, 1} - U_{1, 1})dt + \frac{\beta^2N}{2} q R_{1, 1}dt\right)
    \end{align*}
    Since $Z_t = \sum_{\sigma} \exp(H_t(\sigma))$, 
    \begin{align*}
        dZ_t &= \sum_{\sigma} d\exp(H_t(\sigma))\\
        &= Z_t \inner{dX_t(\sigma) + dY_t(\sigma)}_t\\
        &+ Z_t \inner{ - \frac{\beta^2}{2} (p - q) \sum_i \sigma_i^2 dt + \frac{\beta^2N}{4} (R^2_{1, 1} - U_{1, 1})dt + \frac{\beta^2N}{2} q R_{1, 1}dt}_t
    \end{align*} 
    For the QV, 
    \begin{align*}
        d[Z]_t &= \sum_{\sigma, \tau}d[\exp(H(\sigma)), \exp(H(\tau))]_t\\
        &= Z_t^2 \inner{d[X(\sigma), X(\tau)]_t + d[Y(\sigma), Y(\tau)]}_t\\
        &= Z_t^2 \inner{\frac{\beta^2N}{2}(R^2_{1, 2} - U_{1, 2})dt + \beta^2 q N R_{1, 2}dt}_t
    \end{align*}
\end{proof}
Now, we can write the free energy as 
\[\log Z_1 - \log Z_0 = \int_0^1 d \log Z_t \]
where $\frac{1}{N}\log Z_1$ is the usual free energy. 
\begin{lemma} \label{lem: Free energy 0}
    \[
    \frac{1}{N}\log Z_0 = \log\left(1 +  \sum_{\gamma = 1}^{S} 2\cosh((\beta \sqrt{q}\eta + h)\gamma) \exp(\gamma^2\left(\frac{\beta^2}{2}(p - q) + D\right)) \right)
    \]
\end{lemma}
\begin{proof}
    At time $0$, the spins are independent of each other, and we can have 
    \begin{align}
        \log Z_0 &= \log \left(\sum_{\sigma} \exp(\beta \sqrt{q} \sum_i \sigma_i \eta_i + \frac{\beta^2}{2}(p - q)\sum_i \sigma_i^2 + D\sum_i \sigma^2_i + h\sum_i \sigma_i)\right)\\
        &= N\log\left(1 +  \sum_{\gamma = 1}^{S} 2\cosh((\beta \sqrt{q}\eta + h)\gamma) \exp(\gamma^2(\frac{\beta^2}{2}(p - q) + D)) \right)
    \end{align}
\end{proof}
\begin{lemma} \label{lem: dlog}
\begin{align*}
    d \log Z_t &=\inner{dX_t(\sigma) + \beta \sqrt{q} \sum_i \sigma_i dB_i[t]}_t\\
                &- \inner{\beta \sqrt{q} \sum_i \sigma_i \frac{W_i[t]}{1 - t}dt}_t\\
                & + \beta^2 q N \inner{R_{1, 1} - R_{1, 2}}_t dt\\
                &+ \frac{\beta^2N}{4} \inner{(R_{1, 1} - p)^2 - (R_{1, 2} - q)^2 + q^2 - p^2 + U_{1, 2} - U_{1, 1}  }_t dt
\end{align*}
\end{lemma}
\begin{proof}
    From Lemma \ref{lem: partition fn}
\begin{align*}
    d \log Z_t &= \frac{1}{Z_t}dZ_t - \frac{1}{2}\frac{d[Z]_t}{Z_t^2} \\
    &= \inner{dX_t(\sigma) + \beta \sqrt{q} \sum_i \sigma_i dB_i[t]}_t\\
    &- \inner{\beta \sqrt{q} \sum_i \sigma_i \frac{W_i[t]}{1 - t}dt}_t\\
    &+ \inner{ - \frac{\beta^2}{2} (p - q) NR_{1, 1} dt + \frac{\beta^2N}{4} (R^2_{1, 1} - U_{1, 1})dt + \frac{\beta^2N}{2} q R_{1, 1}dt}_t\\
    &- \frac{1}{2}\inner{\frac{\beta^2N}{2}(R^2_{1, 2} - U_{1, 2})dt + \beta^2 q N R_{1, 2}dt}_t
\end{align*}
For the term involving $W_i[t]$, note that $\frac{W_i[t]}{1 - t} \sim \cc{N}(0, 1)$. From Gaussian integration by part, we get
\begin{align} \label{eq: gip}
    \beta\sqrt{q}\E[\sum_i \inner{\sigma_i}_t \frac{W_i[t]}{1 - t}] = \sum_i \beta\sqrt{q}N\E[\partial_{W_i[t]} \inner{\sigma_i}_t] = \beta^2 q N \E[\inner{R_{1, 1} - R_{1, 2}}_t]
\end{align}
Separate this term so that we can normalized $\beta\sqrt{q}\E[\sum_i \inner{\sigma_i}_t \frac{W_i[t]}{1 - t}]$ later, and rearrange other drift terms 
\begin{align*}
    &\beta^2qN \inner{R_{1, 1} - R_{1, 2}}_t dt\\
    &-\frac{\beta^2 N}{2}\inner{(p - q)R_{1, 1}}dt + \frac{\beta^2 N}{4} \inner{R_{1, 1}^2 - U_{1, 1} - 2qR_{1, 1}}_tdt - \frac{\beta^2 N}{4}\inner{R_{1, 2}^2 - U_{1, 1} - 2qR_{1, 2}}_t dt\\
    =& \beta^2qN \inner{R_{1, 1} - R_{1, 2}}_t dt \\
    &+ \frac{\beta^2 N}{4} \inner{R_{1, 1}^2 - U_{1, 1} - 2pR_{1, 1}}_tdt - \frac{\beta^2 N}{4}\inner{R_{1, 2}^2 - U_{1, 1} - 2qR_{1, 2}}_t dt
\end{align*}
Rearrange gives
\[\frac{\beta^2N}{4} \inner{(R_{1, 1} - p)^2 - (R_{1, 2} - q)^2 + q^2 - p^2 + U_{1, 2} - U_{1, 1}  }_t dt\]
Plug this back into the expression of $d \log Z_t$ gives 
\begin{align*}
    d \log Z_t &=\inner{dX_t(\sigma) + \beta \sqrt{q} \sum_i \sigma_i dB_i[t]}_t\\
                &- \inner{\beta \sqrt{q} \sum_i \sigma_i \frac{W_i[t]}{1 - t}dt}_t\\
                & + \beta^2 q N \inner{R_{1, 1} - R_{1, 2}}_t dt\\
                &+ \frac{\beta^2N}{4} \inner{(R_{1, 1} - p)^2 - (R_{1, 2} - q)^2 + q^2 - p^2 + U_{1, 2} - U_{1, 1}  }_t dt
\end{align*}
\end{proof}
Now we are ready to center the free energy. 
Let $a = \frac{1}{2}$ and
\[N^a F_t := \log Z_t - \E[\log Z_0] - \frac{\beta^2N}{4}(q^2 - p^2)t\]
By the above computation,
\begin{align*}
    N^a(F_t - F_0) &= \int_0^t \inner{dX_t(\sigma) + \beta \sqrt{q} \sum_i \sigma_i dB_i[t]}_t\\
    &- \int_0^t  \inner{\beta \sqrt{q} \sum_i \sigma_i \frac{W_i[t]}{1 - t}dt}_t\\
    & + \int_0^t \beta^2 q N \inner{R_{1, 1} - R_{1, 2}}_t dt]\\
    &+ \int_0^t \frac{\beta^2N}{4} \inner{(R_{1, 1} - p)^2 - (R_{1, 2} - q)^2 + U_{1, 2} - U_{1, 1}  }_t dt
\end{align*}
Now we are ready to prove the main theorem by bounding $\E[\exp(iu F_t)]$
\begin{proof}
    Apply Ito's lemma on $f(x) = \exp(iu x)$ to get 
    \begin{align}
        d \exp(iu F_t) = \exp(iu F_t) \left(iu d F_t - \frac{1}{2}u^2 d [F]_t \right)
    \end{align}
    where, as in the computation of $d[Z]_t$,
    \begin{align*}
    d[F]_t = N^{-2a}\inner{\frac{\beta^2N}{2}(R^2_{1, 2} - U_{1, 2})dt + \beta^2 q N R_{1, 2}dt}_t
    \end{align*}.
   Moreover
    \[\exp(iu F_T) = \exp(iu F_0) + \int_0^T d \exp(iu F_t)\]
    Combine with the definition above, we have
    \begin{align}
       d \exp(iu F_t) & =\exp(iu F_t) \frac{iu}{N^a}\left( \inner{dX_t(\sigma) + \beta \sqrt{q} \sum_i \sigma_i dB_i[t]}_t\right) \label{eq:1}\\ 
       & - \exp(iu F_t) \frac{iu}{N^a} \inner{\beta \sqrt{q} \sum_i \sigma_i \frac{W_i[t]}{1 - t}dt}_t \label{eq:2}\\
       & + \exp(iu F_t) \frac{iu}{N^a}\beta^2 q N \inner{R_{1, 1} - R_{1, 2}}_t dt \label{eq:3}\\
       & + \exp(iu F_t) \frac{iu}{N^a}\frac{\beta^2N}{4} \inner{(R_{1, 1} - p)^2 - (R_{1, 2} - q)^2 + U_{1, 2} - U_{1, 1}  }_t dt \label{eq:4}\\
       &- \exp(iu F_t)\frac{u^2}{2}N^{-2a}\inner{\frac{\beta^2N}{2}(R^2_{1, 2} - U_{1, 2})dt + \beta^2 q N R_{1, 2}dt}_t \label{eq:5}
    \end{align}
    We are interested in the expectation. Since the first term \ref{eq:1} is a martingale, the expectation is $0$. 
    \begin{itemize}
        \item For \ref{eq:2}, we have 
        \begin{align*}
            \E[\exp(iu F_t) \frac{iu}{N^a} \inner{\beta \sqrt{q} \sum_i \sigma_i \frac{W_i[t]}{1 - t}dt}_t] =& \sum_i \E[\partial_{W_i[t]} \left(\exp(iu F_t) \frac{iu}{N^a} \inner{\beta \sqrt{q} \sigma_i dt}_t\right)]\\
            =& \sum_i \E[\partial_{W_i[t]} \left(\exp(iu F_t)\right) \frac{iu}{N^a} \inner{\beta \sqrt{q} \sigma_i dt}_t]\\
            &+ \sum_i \E[\exp(iu F_t)\frac{iu}{N^a} \partial_{W_i[t]}\inner{\beta \sqrt{q} \sigma_i}_t]
        \end{align*}
        The second term cancels with \ref{eq:3} by construction. See \eqref{eq: gip}. The first term becomes 
        \[\sum_i\E[\frac{-u^2}{N^{2a}}\beta^2 q\inner{\sigma_i}_t \exp(iu F_t) \inner{ \sigma_i}_t dt] = \E[\frac{-u^2}{N^{2a}}N\beta^2 q\inner{R_{1, 2}}_t \exp(iu F_t) dt]\]
        Combine gives
        \[\E[\ref{eq:2} + \ref{eq:3}] =\E[\exp(iu F_t)\frac{u^2}{N^{2a}}N\beta^2 q\inner{R_{1, 2}}_t  dt] \]
        \item Combine above formula with \ref{eq:5} gives
        \begin{align*}
            \E[\ref{eq:2} + \ref{eq:3} + \ref{eq:5}]=&-\E\left[\exp(iu F_t) \frac{u^2}{N^{2a}}\frac{\beta^2 N}{4} \inner{R_{1, 2}^2 - U_{1, 2} + 2qR_{1, 2} - 4qR_{1, 2}}_t dt\right]\\
            = &  -\E\left[\exp(iu F_t) \frac{u^2}{N^{2a}}\frac{\beta^2 N}{4} \inner{(R_{1, 2}-q)^2 - U_{1, 2}-q^2}_t dt\right]
        \end{align*}
    \end{itemize}
    Recall that $a = \frac{1}{2}$, by Proposition \ref{prop:ac21} and definition of $U_{1, 2}$, $U_{1, 1}$, 
    \[\E[\inner{(R_{i, j} - Q_{i, j})^2}] = O(\frac{1}{N}) \quad U_{i, j} = O(\frac{1}{N})\],
   this allows us to bound $\E[\exp(iuF_T)] -\E[\exp(iuF_0)]$ by the following 
    \begin{align} \label{eq: final deri}
        \E[\int_0^T d\exp(iuF_t)] = \E[\int_0^T \exp(iuF_t) \frac{u^2\beta^2}{4}q^2dt] + O(\frac{1}{\sqrt{N}})
    \end{align}
    What's left to show is $\E[\exp(iuF_0)]$. By definition, 
    \[F_0 = \frac{1}{\sqrt{N}} \left(\log Z_0 - \E[\log Z_0]\right)\]
    Let $W_1, \cdots W_N$ be i.i.d. samples from the same distribution as $W = \sum_{\gamma = -S}^S \exp((\beta \sqrt{q}\eta + h) \gamma + \gamma^2 \left(\frac{\beta^2}{2}(p - q) + D\right))$ where $\eta \sim \cc{N}(0, 1)$, then 
    \[F_0 = \sqrt{N}\frac{\sum_i \left(\log W_i -  \E[\log W]\right)}{N} \]
    By the central limit theorem, we have 
    \[\E[\exp(iu F_0)] = \exp(-\frac{u^2\Var(\log W)}{2}) + O(\frac{1}{\sqrt{N}})\]
    Combine with \eqref{eq: final deri} gives
    \[\E[\exp(iuF_1)] = \exp(-\frac{u^2}{2} \left(\Var(\log W) - \frac{q^2\beta^2}{2}\right)) + O(\frac{1}{\sqrt{N}})\]
\end{proof}

\bibliographystyle{alpha}
\bibliography{GS}
\end{document}